\documentclass[reqno,a4paper,draft]{amsart}
\usepackage{enumitem}

\setenumerate{label=\textnormal{(\arabic*)}}

\usepackage{amsmath,amssymb,dsfont,verbatim,mathtools,bm,geometry}

\usepackage[latin1]{inputenc}
\usepackage[raggedright]{titlesec}
\usepackage{mathtools}
\usepackage{tikz}
\usepackage{float,subfig}
\usepackage{amsrefs}

\titleformat{\chapter}[display]
{\normalfont\huge\bfseries}{\chaptertitlename\\thechapter}{20pt}{\Huge}
\titleformat{\section}
{\normalfont\Large\bfseries\center}{\thesection}{1em}{}
\titleformat{\subsection}
{\normalfont\large\bfseries}{\thesubsection}{1em}{}
\titleformat{\subsubsection}[runin]
{\normalfont\normalsize\bfseries}{\thesubsubsection}{1em}{}
\titleformat{\paragraph}[runin]
{\normalfont\normalsize\bfseries}{\theparagraph}{1em}{}
\titleformat{\subparagraph}[runin]
{\normalfont\normalsize\bfseries}{\thesubparagraph}{1em}{}
\titlespacing*{\chapter} {0pt}{50pt}{40pt}
\titlespacing*{\section} {0pt}{3.5ex plus 1ex minus .2ex}{2.3ex plus .2ex}
\titlespacing*{\subsection} {0pt}{3.25ex plus 1ex minus .2ex}{1.5ex plus .2ex}
\titlespacing*{\subsubsection}{0pt}{3.25ex plus 1ex minus .2ex}{1.5ex plus .2ex}
\titlespacing*{\paragraph} {0pt}{3.25ex plus 1ex minus .2ex}{1em}
\titlespacing*{\subparagraph} {\parindent}{3.25ex plus 1ex minus .2ex}{1em}

\subjclass[2000]{Primary 16S35; Secondary 16W30}

\newtheorem{theorem}{Theorem}[section]
\newtheorem{lemma}[theorem]{Lemma}
\newtheorem{proposition}[theorem]{Proposition}
\newtheorem{corollary}[theorem]{Corollary}
\newtheorem{remark}[theorem]{Remark}

\theoremstyle{definition}

\def\iproof{\begin{proof}}
\def\fproof{\end{proof}}

\def\td{\widetilde}

\def\la{\langle}
\def\ra{\rangle}
\def\Cx{\mathbb{C}}

\def\iarray{\begin{array}}
\def\farray{\end{array}}
\def\ieqna{\begin{eqnarray*}}
\def\feqna{\end{eqnarray*}}
\def\ieqn{\begin{eqnarray}}
\def\feqn{\end{eqnarray}}
\def\ienu{\begin{enumerate}}
\def\fenu{\end{enumerate}}
\def\iitem{\begin{itemize}}
\def\fitem{\end{itemize}}

\def\td{\widetilde}

\def\la{\langle}
\def\ra{\rangle}

\def\ds{\displaystyle}

\newcommand{\comb}[2]{\left(\!\!\begin{array}{c}#1\\ #2\end{array}\!\!\right)}
\begin{document}

\title{The Groebner basis of a polynomial system}

\author{Christian Valqui}
\address{Christian Valqui\\
Pontificia Universidad Cat\'olica del Per\'u, Secci\'on Matem\'aticas, PUCP,
Av. Universitaria 1801, San Miguel, Lima 32, Per\'u.}

\address{Instituto de Matem\'atica y Ciencias Afines (IMCA) Calle Los Bi\'ologos 245.
Urb San C\'esar. La Molina, Lima 12, Per\'u.}
\email{cvalqui@pucp.edu.pe}

\thanks{Christian Valqui was supported by PUCP-DGI-2013-3036}

\author{Marco Solorzano}
\address{Marco Solorzano\\
Pontificia Universidad Cat\'olica del Per\'u, Secci\'on Matem\'aticas, PUCP,
Av. Universitaria 1801, San Miguel, Lima 32, Per\'u.}
\email{marco.solorzano@pucp.pe}

\subjclass[2010]{Primary 14R15; Secondary 13F20, 11B99}
\keywords{Jacobian, Groebner basis, Catalan numbers}
\begin{abstract}
We compute the Groebner basis of a system of polynomial equations related to the Jacobian conjecture
using a recursive formula for the Catalan numbers.
\end{abstract}

\maketitle

\section{Introduction}
In this paper $K$ is a characteristic zero field and $K[y]((x^{-1}))$ is the algebra of Laurent series
in $x^{-1}$ with  coefficients in $K[y]$.
In a recent article the following theorem was proved~\cite{GGV}*{Theorem 1.9}.
\begin{theorem}\label{principal} The Jacobian conjecture in dimension two is false if and only if
there exist

\begin{itemize}

\smallskip

\item[-] $P,Q\in K[x,y]$ and $C,F\in K[y]((x^{-1}))$,

\smallskip

\item[-] $n,m\in \mathds{N}$ such that $n\nmid m$ and $m\nmid n$,

\smallskip

\item[-] $\lambda_i\in K$ ($i=0,\dots,m+n-2$) with $\lambda_0=1$,

\smallskip

\end{itemize}
such that
\begin{itemize}

\smallskip

\item[-] $C$ has the form
$$
C = x + C_{-1}x^{-1}+ C_{-2}x^{-2} + \cdots \qquad\text{with each $C_{-i}\in K[y]$,}
$$

\smallskip

\item[-] $gr(C)=1$ and $gr(F)=2-n$, where $gr$ is the total degree,

\smallskip

\item[-] $F_+=x^{1-n}y$, where $F_+$ is the term of maximal degree in $x$ of $F$,

\smallskip

\item[-] $C^n=P$ and $Q=\sum_{i=0}^{m+n-2}\lambda_i C^{m-i}+F$.

\smallskip

\end{itemize}
Furthermore, under these conditions $(P,Q)$ is a counterexample to the Jacobian conjecture.
\end{theorem}
Motivated by this result, the authors consider the following slightly more general situation. Let $D$ be a $K$-algebra (in Theorem~\ref{principal} we take $D=K[y]$),
$n,m$ positive integers such that $n\nmid m$ and $n\nmid m$, $(\lambda_i)_{1\le i\le n+m-2}$ a family of elements
in $K$ with $\lambda_0=1$ and $F_{1-n}\in D$ (in Theorem~\ref{principal} we take $F_{1-n}=y$).
A Laurent series in $x^{-1}$ of the form
$$
C = x + C_{-1}x^{-1}+ C_{-2}x^{-2} + \cdots \qquad\text{with $C_{-i}\in D$,}
$$
is \textsl{a solution of the system} $S(n,m,(\lambda_i),F_{1-n})$, if there exist
 $P,Q\in D[x]$ and $F \in D[[x^{-1}]]$, such that
\begin{align}
& F = F_{1-n} x^{1-n} + F_{-n} x^{-n} + F_{-1-n} x^{-1-n} +\cdots,\label{forma de F}\\
&P=C^n\qquad\text{and}\qquad Q=\sum_{i=0}^{m+n-2}\lambda_i C^{m-i}+F.\label{forma de P y de Q}
\end{align}

For example, if $n=2$, then
\begin{align*}
P(\text{\bf{x}})=&C^2=\text{\bf{ x}}^2+2 C_{-1}+ 2 C_{-2}\text{\bf{ x}}^{-1} + (C_{-1}^2+
2 C_{-3})\text{\bf{ x}}^{-2} +
(2C_{-1}C_{-2}+2 C_{-4})\text{\bf{ x}}^{-3}\\
&+ (C_{-2}^2+2C_{-1}C_{-3}+2 C_{-5})\text{\bf{ x}}^{-4}+(2 C_{-2} C_{-3} + 2 C_{-1} C_{-4} +
2 C_{-6})\text{\bf{ x}}^{-5}+\dots,
\end{align*}
and the condition $C^2\in K[x]$ translates into the following conditions on $C_{-k}$:
\begin{align*}
0=(C^2)_{-1}=& \ 2 C_{-2},\\
0=(C^2)_{-2}=& \ C_{-1}^2 + 2 C_{-3},\\
0=(C^2)_{-3}=& \ 2 C_{-1} C_{-2} + 2 C_{-4},\\
0=(C^2)_{-4}=& \ C_{-2}^2 + 2 C_{-1} C_{-3} + 2 C_{-5},\\
0=(C^2)_{-5}=& \ 2 C_{-2} C_{-3} + 2 C_{-1} C_{-4} + 2 C_{-6},\\
0=(C^2)_{-6}=& \ C_{-3}^2 + 2 C_{-2} C_{-4} + 2 C_{-1} C_{-5} + 2 C_{-7},\\
0=(C^2)_{-7}=& \ 2 C_{-3} C_{-4} + 2 C_{-2} C_{-5} + 2 C_{-1} C_{-6} + 2 C_{-8},\\
0=(C^2)_{-8}=& \ C_{-4}^2 + 2 C_{-3} C_{-5} + 2 C_{-2} C_{-6} + 2 C_{-1} C_{-7} + 2 C_{-9},\\
\vdots
\end{align*}
In general, the condition $P(x)=C^n\in K[x]$ yields equations $(C^n)_{-k}=0$, whereas the condition
$Q(x)=\sum_{i=0}^{m+n-2}\lambda_i C^{m-i}+F\in K[x]$ gives us the equations
$\left(\sum_{i=0}^{m+n-2}\lambda_i C^{m-i}+F\right)_{-k}=0$, where we note that $F_{-k}=0$ for $k=1,\dots,n-2$.

It is easy to see (e.g.~\cite[Remark 1.13]{GGV}) that the first $m+n-2$ coefficients determine
the others, i.e.,
the coefficients $C_{-1},\dots,C_{-m-n+2}$ determine univocally the coefficients $C_{-k}$ for $k>m+n-2$.
Moreover, the $F_{-k}$ for $k>n-1$ depend only on $F_{1-n}$ and $C$.
Consequently, having a solution $C$ to the system $S(n,m,(\lambda_i),F_{1-n})$
is the same as having a solution  $(C_{-1},\dots,C_{-m-n+2})$ to the system
\begin{equation}
\begin{aligned}
&E_k:=(C^n)_{-k}  = 0, && \text{for $k=1,\dots, m-1$,}\\
&E_{m-1+k}:=\left(\sum_{i=0}^{m+n-2}\lambda_i C^{m-i}\right)_{-k}  = 0, &&\text{for $k=1,\dots,n-2$,}\\
&E_{m+n-2}:=\left(\sum_{i=0}^{m+n-2}\lambda_i C^{m-i}\right)_{1-n} + F_{1-n} = 0,
\end{aligned}\label{sistema de ecuaciones}
\end{equation}
with $m+n-2$ equations $E_k$ and $m+n-2$ unknowns $C_{-k}$.

In order to understand the solution set of this system, it would be very helpful to find a Groebner basis
for the ideal generated by the polynomials $E_k$ in $D[C_{-1},\dots,C_{m+n-2}]$.
In this paper we compute such a Groebner basis of~\eqref{sistema de ecuaciones} in a very particular case:
we assume
$n=2$, $m=2r+1$ for some $r>0$, and $\lambda_i=0$
for $i>0$. Moreover we consider $D=\mathds{C}[y]$ and $F_{1-n}=y$, as in Theorem~\ref{principal}.

\section{Computation of a Groebner basis for $I_{2r}$}
\setcounter{equation}{0}
Assume
$n=2$, $m=2r+1$ for some $r>0$, and $\lambda_i=0$
for $i>0$. Set also $D=\mathds{C}[y]$ and $F_{1-n}=y$.

Then the system~\eqref{sistema de ecuaciones} reads
\begin{eqnarray}
\begin{array}{rcl}
E_i&=&\left\{\iarray{ll}
(C^2)_{-i},& i=1,\ldots,2r\\
(C^{2r+1})_{-1}+y,& i=2r+1,
\farray\right.\\
\end{array}\label{eqGGV}
\end{eqnarray}
where $(C^2)_{-i}$ denotes the coefficient of $x^{-i}$ in the Laurent series $C^2$. Explicitly, the
polynomials $E_i$ are given by
\ieqn
E_1&:=&2C_{-2},\nonumber\\
E_2&:=&2C_{-3}+C_{-1}^2,\nonumber\\
E_3&:=&2C_{-4}+2C_{-2}C_{-1},\nonumber\\
E_4&:=&2C_{-5}+2C_{-3}C_{-1}+C_{-2}^2,\nonumber\\
E_5&:=&2C_{-6}+2C_{-2}C_{-3}+2C_{-4}C_{-1},\label{ecuacion01}\\
E_6&:=&2C_{-7}+2C_{-5}C_{-1}+2C_{-4}C_{-2}+C_{-3}^2,\nonumber\\
&\vdots&\nonumber\\
E_{2r-1}&:=&2C_{-2r}+2C_{-2}C_{-2r+3}+2C_{-4}C_{-2r+5}+\cdots+2C_{-2r+4}C_{-3}+2C_{-2r+2}C_{-1},\nonumber\\
E_{2r}&:=&2C_{-2r-1}+2C_{-2r+1}C_{-1}+2C_{-2r+2}C_{-2}+\cdots+
C_{-r}^2,\nonumber\\
E_{2r+1}&:=&(C^{2r+1})_{-1}
+y.\nonumber
\feqn

Each $E_i$ is a polynomial in the ring $\Cx[C_{-1},C_{-2},\ldots,C_{-2r-1},y]$, and the $2r+1$ polynomials
yield the ideal
\ieqna
I&=&\la E_1,\ldots,E_{2r},E_{2r+1}\ra.
\feqna
Our goal  is to find a Groebner basis for the ideal $I$. However, in this section we will
only compute a Groebner basis
$(\td{E}_1,\td{E}_2,\ldots,\td{E}_{2r-1},\td{E}_{2r})$ for the ideal $I_{2r}:=\la
E_1,E_2,\ldots,E_{2r-1},E_{2r}\ra$.
Note that for $i=1\dots, 2r$ we have
\begin{equation}\label{E sub i}
E_{i}=2C_{-i-1}+\sum_{k=1}^{i-1} C_{-k}C_{k-i}.
\end{equation}

\noindent We first replace the odd numbered polynomials $E_1,E_3,E_5,E_7,\ldots,E_{2r-1}$ by
new polynomials $\td{E}_1,\td{E}_3,\td{E}_5,\td{E}_7,\ldots,\td{E}_{2r-1}$ defined by
\ieqn
\td{E}_1&:=&C_{-2}=\frac12E_1,\nonumber\\
\td{E}_3&:=&C_{-4}=\frac12E_3-\td{E}_1C_{-1},\nonumber\\
\td{E}_5&:=&C_{-6}=\frac12E_5-\td{E}_1C_{-3}-\td{E}_3C_{-1},\label{ecuacionPar}\\
\td{E}_7&:=&C_
{-8}=\frac12E_7-\td{E}_1C_{-5}-\td{E}_3C_{-3}-\td{E}_5C_
{-1},\nonumber\\
\td{E}_9&:=&C_{-10}=\frac12E_9-\td{E}_1C_{-7}-\td{E}_3C_{-5}-\td{E}_5C_
{-3}-\td{E}_7C_{-1},\nonumber\\
&\vdots&\nonumber\\
\td{E}_{2r-1}&:=&C_
{-2r}=\frac12E_{2r-1}-\sum_{i=1}^{r-1}\td{E}_{2i-1}C_
{-2(r-i)+1}.\nonumber
\feqn
\begin{remark}
We have
\begin{equation}\label{impares}
    \la E_1,E_3,\ldots,E_{2r-1}\ra=\la \td{E}_1,\td{E}_3,\ldots,\td{E}_{2r-1}\ra.
\end{equation}
In fact, if we define $\td{I}_k^{odd}:= \la \td{E}_1,\td{E}_3,\ldots,\td{E}_{2k-1}\ra$,
then~\eqref{ecuacionPar} clearly implies
\begin{equation}\label{caso impar}
E_{2i+1}-2\td{E}_{2i+1}\in\td{I}_{i}^{odd},
\end{equation}
and so we get $\la E_1,E_3,\ldots,E_{2i+1}\ra\subset \la \td{E}_1,\td{E}_3,\ldots,\td{E}_{2i+1}\ra$
for all $i$.
Using induction one sees that we also have
$\la \td{E}_1,\td{E}_3,\ldots,\td{E}_{2r-1}\ra\subset\la E_1,E_3,\ldots,E_{2r-1}\ra$,
as desired.
\end{remark}

\noindent The next proposition deals with
$E_2,E_4,E_6,\ldots,E_{2r}$, the first $r$  even numbered polynomials.

\begin{proposition}
\label{caso par}
For all $j\in\mathds{N}$ there exists $\lambda_j$ such that for
$\td{E}_{2j}:=C_{-2j-1}+\lambda_jC_{-1}^{j+1}$
we have
\begin{equation}\label{formula}
C_{-2j-1}+\lambda_jC_{-1}^{j+1}-\ds{\frac12}E_{2j}\in\td{I}_{2j-1}:=
\la\td{E}_1,\td{E}_2,\ldots,\td{E}_{2j-2},\td{E}_{2j-1}\ra.
\end{equation}
Moreover, if we set $\lambda_0=-1$ and $E_0=\td{E}_0:=0$, then for $j>0$, $\lambda_j$ is given by
\begin{equation}\label{lambda}
    \lambda_j:=\frac12 \left(\sum_{k=0}^{j-1} \lambda_k \lambda_{j-k-1} \right).
\end{equation}

\end{proposition}
\begin{proof}
We proceed by induction on $j$.
For $j=0$ clearly~\eqref{formula} is satisfied. For $j=1$, with $\ds{\lambda_1=\frac12}$ calculated
by~\eqref{lambda},
we have
$$\ds{C_{-3}+\frac12C_{-1}^2-\frac12E_2=0\in\la\td{E}_1\ra},$$
as desired.

From~\eqref{E sub i} we have
\begin{align*}
E_{2j}=&2C_{-2j-1}+\sum_{k=1}^{2j-1} C_{-k}C_{k-2j}\\
=&2C_{-2j-1}+\sum_{k=0}^{j-1} C_{-2k-1}C_{2k+1-2j}
+\sum_{k=1}^{j-1} C_{-2k}C_{2k-2j},
\end{align*}
and clearly $\sum_{k=1}^{j-1} C_{-2k}C_{2k-2j}\in \td{I}_{2j-1}$. Therefore we get
\begin{equation}\label{suma}
C_{-2j-1}-\ds{\frac12}E_{2j}\in-\frac12 \left(\sum_{k=0}^{j-1} C_{-2k-1}C_{2k+1-2j}\right)+\td{I}_{2j-1}.
\end{equation}
By the induction hypothesis, for $0\le k\le j-1$ there exist $\lambda_{k}$ and $\lambda_{j-k-1}$ such that
$$
C_{-2k-1}=-\lambda_k C_{-1}^{k+1}+\td{E}_{2k} \quad\text{and}\quad
C_{2k+1-2j}=-\lambda_{j-k-1} C_{-1}^{j-k}+\td{E}_{2(j-k-1)};
$$
hence
$$
C_{-2k-1}C_{2k+1-2j}\in \lambda_k \lambda_{j-k-1} C_{-1}^{j+1}+\td{I}_{2j-1}.
$$
From~\eqref{suma} we obtain
$$
C_{-2j-1}-\ds{\frac12}E_{2j}\in-\frac12 \left(\sum_{k=0}^{j-1} \lambda_k \lambda_{j-k-1} \right)
C_{-1}^{j+1}+\td{I}_{2j-1},
$$
from which~\eqref{formula} follows with
$\lambda_j=\frac12 \left(\sum_{k=0}^{j-1} \lambda_k \lambda_{j-k-1} \right)$, as desired.
\fproof

\begin{corollary}\label{teorem}
We have
\begin{equation}\label{impares}
    \la E_1,E_2,\ldots,E_{2r}\ra=\la \td{E}_1,\td{E}_2,\ldots,\td{E}_{2r}\ra.
\end{equation}
\end{corollary}
\begin{proof}
In fact, if we define $\td{I}_k:= \la \td{E}_1,\td{E}_2,\ldots,\td{E}_{k}\ra$,
then~\eqref{caso impar} and Proposition~\ref{caso par} clearly imply
$$
E_{k+1}-2\td{E}_{k+1}\in\td{I}_{k},
$$
and so we get $\la E_1,E_2,\ldots,E_{k+1}\ra\subset \la \td{E}_1,\td{E}_2,\ldots,\td{E}_{k+1}\ra$ for all $k$.
 Since $\la E_1\ra=\la\td{E}_1\ra$, using induction one also obtains
$\la \td{E}_1,\td{E}_2,\ldots,\td{E}_{k}\ra\subset\la E_1,E_2,\ldots,E_{k}\ra$,
as desired.
\end{proof}
We can replace the system~\eqref{ecuacion01} with the following set of equations.

\begin{align*}
\td{E}_1=C_{-2}=0,\quad\td{E}_3=C_{-4}=0,&\quad \dots \quad\td{E}_{2r-1}=C_{-2r}=0,\\
\td{E}_2=C_{-3}+\lambda_{1}C_{-1}^2=0,\quad \td{E}_4=C_{-5}+\lambda_{2}C_{-1}^3=0,&\quad\dots\quad
\td{E}_{2r}=C_{-2r-1}+\lambda_{r}C_{-1}^{r+1}=0,\\
E_{2r+1}=(C^{2r+1})_{-1}
+y=0.&
\end{align*}

\begin{proposition}\label{proposicion03}
If we fix the lex order with $C_{-2r-1}>C_{-2r}>\cdots>C_{-3}>C_{-2}>C_{-1}>y$, then
$G_{2r}=(\td{E}_1,\td{E}_2,\ldots,\td{E}_{2r-1},\td{E}_{2r})$ is a Groebner  basis of the ideal
$$
\td{I}_{2r}=\la \td{E}_1,\td{E}_2,\ldots,\td{E}_{2r-1},\td{E}_{2r}\ra
$$
\end{proposition}
\iproof
We first compute the $S$-polynomials of $G_{2r}$, and prove that they satisfy
$\overline{S(\td{E}_{i},\td{E}_{j})}^{G_{2r}}=0$ for all
 $1\leq i,j\leq 2r$.\\

Consider first the $S$-polynomial of an even-numbered polynomial and an odd numbered polynomial.
So look now at $\td{E}_{2s-1}$ and $\td{E}_{2t}$, with $1\leq s,t\leq r$. We have
\ieqna
S(\td{E}_{2s-1},\td{E}_{2t})&=&C_{-2t-1}C_{-2s}-C_{-2s}(C_{-2t-1}+\lambda_{t}C_{-1}^{t+1})\\
&=&-\lambda_{t}C_{-1}^{t+1}C_{-2s}\\
&=&-\lambda_{t}C_{-1}^{t+1}\td{E}_{2s-1},
\feqna
and so $\overline{S(\td{E}_{2s-1},\td{E}_{2t})}^{G_{2r}}=0$, for all $1\leq s,t\leq r$.

In the case that $i,j$ are both odd, we take $\td{E}_{2s-1},\td{E}_{2t-1}$, with $1\leq s,t\leq r$.
Then we have
\ieqna
S(\td{E}_{2s-1},\td{E}_{2t-1})&=&C_{-2t}C_{-2s}-C_{-2s}C_{-2t}=0,
\feqna
and trivially we get $\overline{S(\td{E}_{2s-1},\td{E}_{2t-1})}^{G_{2r}}=0$, for all
$1\leq s,t\leq r$.\\

In the last case, when $i,j$ are even, consider $\td{E}_{2s},\td{E}_{2t}$, with $1\leq s,t\leq r$. Then
we have
\ieqna
S(\td{E}_{2s},\td{E}_{2t})&=&C_{-2t-1}(C_{-2s-1}+\lambda_{s}C_{-1}^{s+1})-C_{-2s-1}(C_{-2t-1}+
\lambda_{t}C_{-1}^{t+1})\\
&=&\lambda_{s}C_{-1}^{s+1}C_{-2t-1} -\lambda_{t}C_{-1}^{t+1}C_{-2s-1}.
\feqna
Now we divide $S(\td{E}_{2s},\td{E}_{2t})$ by $G_{2r}$. If $C_{-2t-1}>C_{-2s-1}$, then the leading
term is
$$
lt(S(\td{E}_{2s},\td{E}_{2t}))=\lambda_{s}C_{-1}^{s+1}C_{-2t-1}
$$
and the first division step yields
$$
S(\td{E}_{2s},\td{E}_{2t})=\lambda_sC_{-1}^{s+1}\td{E}_{2t}+R_1,
$$
with $R_1=-\lambda_{s}\lambda_{t}C_{-1}^{s+t+2}-\lambda_{t}C_{-1}^{t+1}C_{-2s-1}$. But
continuing the division algorithm we obtain
$$
R_1=-\lambda_tC_{-1}^{t+1}\td{E}_{2s},
$$
and hence $\overline{S(\td{E}_{2s},\td{E}_{2t})}^{G_{2r}}=0$ in this case.
The case $C_{-2s-1}>C_{-2t-1}$ is similar,
so we get $\overline{S(\td{E}_{2t},\td{E}_{2s})}^{G_{2r}}=0$ for all
 $1\leq s,t\leq r$.
\fproof
From Corollary~\ref{teorem} and Proposition~\ref{proposicion03} we conclude that
$(\td{E}_1,\td{E}_2,\ldots,\td{E}_{2r-1},\td{E}_{2r})$ is a Groebner basis
for the ideal $\la E_1,E_2,\ldots,E_{2r-1},E_{2r}\ra$.\\

\section{A recursive formula for the Catalan numbers and a Groebner basis}
\setcounter{equation}{0}
In this last section we will determine a Groebner basis for the ideal $I$ given by the complete
system~\eqref{eqGGV}.
  In order to achieve this, we need
to establish additional properties of the $\lambda_j$'s which are closely related to the ubiquitous Catalan numbers.
\begin{lemma}\label{catalan}
For all $j\ge 0$ the equality
\begin{equation}\label{c de lambda}
c_j=(-1)^{j+1}2^j\lambda_j
\end{equation}
holds, where $c_j$ are the Catalan numbers given by
$$
c_j=\frac{1}{j+1}\binom{2j}{j}.
$$
\end{lemma}
\begin{proof}
The Catalan numbers are uniquely determined (see e.g.~\cite{K}*{p.117 (5.6)})
 by $c_0=1$ and the recursive relation
$$
c_r=\sum_{j=0}^{r-1}c_jc_{r-1-j}.
$$
Set $d_j=(-1)^{j+1}2^j\lambda_j$. Then $d_0=1$, since $\lambda_0=-1$, and
 equality~\eqref{lambda} gives us
\begin{align*}
d_j=&  (-1)^{j+1}2^j\lambda_j\\
&= (-1)^{j+1}2^j\frac12 \left(\sum_{k=0}^{j-1} \lambda_k \lambda_{j-k-1} \right)\\
  =& \sum_{k=0}^{j-1} \left((-1)^{k+1}2^k\lambda_k\right)\left( (-1)^{j-k}2^{j-1-k}\lambda_{j-k-1}\right)\\
  =&\sum_{k=0}^{j-1}d_kd_{j-1-k},
\end{align*}
and hence $d_j=c_j$ for all $j$, as desired.
\end{proof}

Now we prove a recursive formula for the  Catalan numbers.
\begin{proposition}
The Catalan numbers satisfy the following formula
\begin{equation}\label{catalan}
    (2r+1)\frac{c_r}{2^{2r}}=\sum_{j=0}^r(-1)^j\binom{r}{j}\frac{c_j}{2^{2j}}.
\end{equation}
Consequently,  $\lambda_r$ satisfies
\begin{equation}\label{formula para lambda}
    (2r+1)(-1)^{r+1}\lambda_r=\sum_{j=0}^r\left(\iarray{c} r\\j\farray\right)2^{r-j}(-\lambda_j).
\end{equation}
\end{proposition}

\begin{proof} Replacing $c_j$ in~\eqref{catalan} and using~\eqref{c de lambda}
yields~\eqref{formula para lambda}, hence, it suffices to prove only~\eqref{catalan}.
For that, we replace $c_j$ by $\frac{1}{j+1}\binom{2j}{j}$ on the righthand side of~\eqref{catalan}
and use the equalities
$$
\comb{-1/2}{j}=\frac{(-1)^j}{2^{2j}}\comb{2j}{j}\quad\text{and}\quad
\comb{r+1/2}{r}=\frac{(2r+1)}{2^{2r}}\comb{2r}{r}.
$$
Then we have
\ieqna
\sum_{j=0}^r (-1)^j\comb{r}{j} \frac{c_j}{2^{2j}}
&=&\sum_{j=0}^r\frac{(-1)^j}{2^{2j}}\comb{2j}{j}\cdot\frac{1}{(j+1)}\comb{r}{j}\\
&=&\sum_{j=0}^r\comb{-1/2}{j}\frac{1}{r+1}\comb{r+1}{j+1}\\
&=&\frac{1}{(r+1)}\sum_{j=0}^r\comb{-1/2}{j}\cdot \comb{r+1}{r-j}\\
&=&\frac{1}{(r+1)}\comb{r+1/2}{r}\\
&=&\frac{1}{(r+1)}\frac{(2r+1)}{2^{2r}}\comb{2r}{r}\\
&=&(2r+1)\frac{c_r}{2^{2r}}.
\feqna
The second equality follows from the relation
$\displaystyle\frac{1}{j+1}\comb{r}{j}=\displaystyle\frac{1}{(r+1)}\comb{r+1}{j+1}$ and the fourth equality
from $\comb{\alpha+\beta}{r}=\displaystyle\sum_{j=0}^{r}\comb{\alpha}{j}\comb{\beta}{r-j}$, valid for all
$\alpha,\beta\in\mathds{Q}$.
\end{proof}
\begin{proposition}\label{central}
Let $I_{2r}=\la E_1,E_2,\ldots,E_{2r}\ra$. Then
$$
(C^{2r+1})_{-1}\in\mu_r C_{-1}^{r+1}+I_{2r},
$$
for $\mu_r=\frac{2r+1}{(r+1)2^r}\binom{2r}{r}$.
\end{proposition}
\begin{proof}
By definition we have
$$
(C^{2r+1})_{-1}=[(C^2)^rC]_{-1}=\sum_{j=-2}^{2r}[(C^2)^r]_jC_{-j-1},
$$
since $C_{-j-1}=0$ for $j<-2$ and $[(C^2)^r]_j=0$ for $j>2r$.

But we also have
$[(C^2)^r]_j=\sum_{i_1+\dots+i_r=j}(C^2)_{i_1}\dots(C^2)_{i_r}$. We claim that $i_k\ge -2r$. In fact,
as $i_j\le 2$, then so we have
$$
i_1+\dots+i_{k-1}+i_{k+1}+\dots+i_r\le 2(r-1),
$$
and $j=i_k+(i_1+\dots+i_{k-1}+i_{k+1}+\dots+i_r)\le 2(r-1)+i_k$ as well. Therefore
we get $i_k\ge j-2r+2\ge -2r$, since $j\ge -2$.

By definition we have $E_i=(C^2)_{-i}$ for $i=1,\dots,2r$. Consequently we obtain
$$
(C^2)_{i_1}\dots(C^2)_{i_r}\in I_{2r},\quad\text{if some $i_k<0$.}
$$
It follows that
\begin{equation}\label{suma positiva}
[(C^2)^r]_j\in \sum_{{\substack{i_1+\dots+i_r=j\\i_k\ge 0 }}}(C^2)_{i_1}\dots(C^2)_{i_r}+I_{2r}
=[(x^2+2C_{-1})^r]_{j}+I_{2r},
\end{equation}
holds, since $C^2=x^2+2C_{-1}+(C^2)_{-1}x^{-1}+(C^2)_{-2}x^{-2}+(C^2)_{-3}x^{-3}+\dots$.
But we also have
$$
(x^2+2C_{-1})^r=\sum_{k=0}^r \binom{r}{k}(2C_{-1})^{r-k}x^{2k},
$$
and so
$$
[(x^2+2C_{-1})^r]_{j}=\left\{\begin{array}{ll}\binom{r}{k}(2C_{-1})^{r-k}&\text{if $j=2k$}\\
0,&\text{if $j=2k+1$}.\end{array}\right.
$$
We arrive at
$$
(C^{2r+1})_{-1}\in \sum_{k=0}^r \binom{r}{k}(2C_{-1})^{r-k}C_{-2k-1}+I_{2r}.
$$
Note that by Proposition~\ref{caso par} we have
$$
C_{-2k-1}=\td{E}_{2k}-\lambda_kC_{-1}^{k+1}\in -\lambda_kC_{-1}^{k+1}+I_{2r},
$$
so we obtain
\begin{align*}
(C^{2r+1})_{-1}&\in \sum_{k=0}^r \binom{r}{k}(2C_{-1})^{r-k}(-\lambda_kC_{-1}^{k+1})+I_{2r}\\
&=
\left(\sum_{k=0}^r \binom{r}{k}2^{r-k}(-\lambda_k)\right)(C_{-1})^{r+1}+I_{2r},
\end{align*}
and the formula for $\mu_r$ follows now from~\eqref{c de lambda} and~\eqref{formula para lambda}.
\end{proof}

\begin{corollary}\label{ideales iguales}
For $\td{E}_{2r+1}:=\mu_r (C_{-1})^{r+1}+y$ we have
$$
\la E_1,E_2,\ldots,E_{2r-1},E_{2r},E_{2r+1}\ra=
\la \td{E}_1,\td{E}_2,\ldots,\td{E}_{2r-1},\td{E}_{2r},\td{E}_{2r+1}\ra.
$$
\end{corollary}
\begin{proof}
By Proposition~\ref{central} we have $E_{2r+1}-\td{E}_{2r+1}=(C^{2r+1})_{-1}-\mu_r C_{-1}^{r+1}\in I_{2r}$,
hence the corollary follows from Corollary~\ref{teorem}.
\end{proof}
Now we can state our main result.
\begin{theorem}\label{proposicion10}
If we fix the lex order with $C_{-2r-1}>C_{-2r}>\cdots>C_{-3}>C_{-2}>C_{-1}>y$, then
$G_{2r+1}=(\td{E}_1,\td{E}_2,\ldots,\td{E}_{2r},\td{E}_{2r+1})$ is a Groebner basis for the ideal
$$
I=\la E_1,E_2,\ldots,E_{2r-1},E_{2r},E_{2r+1}\ra.
$$
\end{theorem}
\iproof
By Corollary~\ref{ideales iguales} it suffices to prove that the division of the
$S$-polynomials $S(\td{E}_i,\td{E}_j)$ by $G_{2r+1}$ is zero.
If $i,j\le 2r$, then the division algorithm yields the same quotients and remainders as in
Proposition~\ref{proposicion03}, since the remainders become zero before one has to divide by
$\td{E}_{2r+1}$. Note that
$lt(\td{E}_{2r+1})=\mu_r (C_{-1})^{r+1}$, since $\mu_r\ne 0$.
It remains to divide the $S$-polynomials $S(\td{E}_i,\td{E}_{2r+1})$ by $G_{2r+1}$.
We first consider the case $i=2t-1$ for some $t=1,\ldots,r$. We get
\ieqna
S(\td{E}_{2t-1},\td{E}_{2r+1})&=&\frac{C_{-2t}C_{-1}^{r+1}}{C_{-2t}}(C_{-2t})-
\frac{C_{-2t}C_{-1}^{r+1}}{\mu_r C_{-1}^{r+1}}(\mu_r C_{-1}^{r+1}+y)\\
&=&-\frac{1}{\mu_r}yC_{-2t},
\feqna
for all $t=1,\ldots,r$. The first division step yields
$S(\td{E}_{2t-1},\td{E}_{2r+1})=-\frac{1}{\mu_r}\td{E}_{2t-1}$, hence we obtain
$\overline{S(\td{E}_{2t-1},\td{E}_{2r+1})}^{G_{2r+1}}=0$, for all $t=1,\ldots,r$.\\

Now for the $S$-polynomials of $\td{E}_{2t}$ and $\td{E}_{2r+1}$, for some $t=1,\ldots,r$, we have
\ieqna
S(\td{E}_{2t},\td{E}_{2r+1})&=&\frac{C_{-2t-1}C_{-1}^{r+1}}{C_{-2t-1}}(C_{-2t-1}+\lambda_t
C_{-1}^{t+1})-\frac{C_{-2t-1}C_{-1}^{r+1}}{\mu_r C_{-1}^{r+1}}(\mu_r C_{-1}^{r+1}+y)\\
&=&\lambda_t C_{-1}^{r+t+2}-\frac{1}{\mu_r}C_{-2t-1}y.
\feqna
with leading term
$$
lt(S(\td{E}_{2t},\td{E}_{2r+1}))=-\frac{1}{\mu_r}C_{-2t-1}y.
$$
We divide $S(\td{E}_{2t},\td{E}_{2r+1})$ by $G_{2r+1}$, and the first division step gives us
$$
S(\td{E}_{2t},\td{E}_{2r+1})=-\frac{1}{\mu_r}y\td{E}_{2t}+R_1
$$
with $R_1=\lambda_t C_{-1}^{r+t+2}+\frac{\lambda_t}{\mu_r} yC_{-1}^{t+1}$. Finally we note that
$R_1=\frac{\lambda_t}{\mu_r} C_{-1}^{t+1}\td{E}_{2r+1}$, in order
to obtain $\overline{S(\td{E}_{2t},\td{E}_{2r+1})}^{G_{2r+1}}=0$, for all $t=1,\ldots,r$.
This concludes the proof.
\fproof
We give explicitly the Groebner basis $G_{2r+1}=(\td{E}_1,\td{E}_2,\ldots,\td{E}_{2r},\td{E}_{2r+1})$ of $I$ as
\begin{align*}
\td{E}_1=C_{-2},&\quad\td{E}_3=C_{-4},&&\quad \dots \quad\td{E}_{2r-1}=C_{-2r}\\
\td{E}_2=C_{-3}+\lambda_{1}C_{-1}^2,&\quad \td{E}_4=C_{-5}+\lambda_{2}C_{-1}^3,&&\quad\dots\quad
\td{E}_{2r}=C_{-2r-1}+\lambda_{r}C_{-1}^{r+1}\\
\td{E}_{2r+1}=\mu_r (C_{-1})^{r+1}+y,&
\end{align*}
with
$$
\mu_r=\frac{2r+1}{(r+1)2^r}\binom{2r}{r}\quad\text{and}\quad
\lambda_j=\frac{(-1)^{j+1}}{(j+1)2^j}\binom{2j}{j}.
$$

\end{document}